\documentclass[12pt]{amsart}

\usepackage{amssymb}

\usepackage{enumerate}

\usepackage{graphicx}

\makeatletter
\@namedef{subjclassname@2010}{%
  \textup{2010} Mathematics Subject Classification}
\makeatother

\newtheorem{thm}{Theorem}[section]

\newtheorem{lem}[thm]{Lemma}
\newtheorem{prop}[thm]{Proposition}

\theoremstyle{definition}

\numberwithin{equation}{section}

\newcommand{\be}{\begin{equation}}
\newcommand{\ee}{\end{equation}}
\newcommand{\ben}{\begin{enumerate}}
\newcommand{\een}{\end{enumerate}}
\newcommand{\beq}{\begin{eqnarray}}
\newcommand{\eeq}{\end{eqnarray}}
\newcommand{\beqn}{\begin{eqnarray*}}
\newcommand{\eeqn}{\end{eqnarray*}}

\newcommand{\e}{{\epsilon}}

\newcommand{\ssfm}{spherically symmetric Finsler metric}
\newcommand{\cfc}{constant flag curvature}

\begin{document}


\baselineskip=17pt



\title{Spherically Symmetric Finsler Metrics With Constant Ricci and Flag Curvature}

\author[E. \d{S}engelen Sev\.{i}m]{Esra \d{S}engelen Sev\.{i}m}
\address{DEPARTMENT OF MATHEMATICAL SCIENCES\\
\.{I}STANBUL B\.{I}LG\.{I} UNIVERSITY\\
ESK\.{I} S\.{I}LAHTARA\u{G}A ELEKTR\.{I}K SANTRAL\.{I}\\
KAZIM KARABEK\.{I}R CAD. NO: 2/13\\
34060 EY\"UP, \.{I}STANBUL, TURKEY}
\email{esra.sengelen@bilgi.edu.tr}

\author[Z. Shen]{Zhongmin Shen}
\address{DEPARTMENT OF MATHEMATICAL SCIENCES\\
INDIANA UNIVERSITY-PURDUE UNIVERSITY\\
INDIANAPOLIS, IN 46202-3216, USA}
\email{zshen@math.iupui.edu}

\author[S. \"Ulgen]{Sema\.{i}l \"Ulgen}
\address{DEPARTMENT OF INDUSTRIAL ENGINEERING\\
ANTALYA INTERNATIONAL  UNIVERSITY\\
\c{C}IPLAKLI MAHALLES\.{I} FARAB\.{I} CADDES\.{I} NO: 23 07190\\
D\"O\d{S}EMEALTI, ANTALYA, TURKEY}
\email{sulgen@antalya.edu.tr}

\date{26 March 2015}

\begin{abstract}
Spherically symmetric metrics form a rich and important class of metrics. Many well-known Finsler metrics of constant flag curvature can be  locally expressed as a spherically symmetric metric on $R^n$.  In this paper, we study spherically symmetric  metrics with constant Ricci curvature and constant flag curvature.
\end{abstract}

\subjclass[2010]{Primary 53B40}

\keywords{Spherically symmetric metrics,  constant Ricci curvature, constant flag curvature}

\maketitle

\section{Introduction}
It is one of important problems in Finsler geometry to study and characterize Finsler metrics with constant flag curvature or  constant Ricci curvature. Let $R^{\ i}_{j \ kl}$ denote the Riemann curvature tensor of the Berwald connection and $R^i_{\ k}:= R^{\i}_{j \ kl}y^jy^l$. The Ricci curvature $Ric$ is defined as $Ric = R^m_{\ m}$.  A Finsler metric $F$ is said to be of {\it constant flag curvature}  if
\[  R^i_{\ k} = K  \{ F^2 \delta^i_k - g_{kl}y^l  y^i \},\]
and it is aid to be of {\it constant Ricci curvature } if
\[ Ric = (n-1) K F^2.\]

Many Finsler metrics of constant flag curvatures can be locally expressed  on a ball $B^n(\rho)\subset R^n$  in the following form
\[   F = |y|\phi(r,s), \ \ \ \ r=|x|, \ s= \frac{\langle x, y \rangle}{|y|}, \ \ \  \ y\in T_x B(\rho)\equiv R^n, \]
where $\phi =\phi(r, s)$ is a positive smooth function defined on $[0, \rho)\times (-\rho,\rho)$.  Finsler  metrics in this form are called {\it spherically symmetric metrics}.  For example,
 the  well-known Funk metric on $B^n(1)\subset R^n$ is projectively flat with constant flag curvature $K=-1/4$. $\phi=\phi(r,s)$  is given by
\be
\phi = \frac{ \sqrt{ 1-(r^2-s^2) } + s}{1-r^2}. \label{phi_Funk}
\ee
Using the  above Funk metric, one can construct another projectively flat metric on $B^n(1)$ with zero flag curvature  $K=0$ (due to L. Berwald). \\

$\phi=\phi(r,s)$ is given by
\be
\phi= \frac{(\sqrt{ 1-(r^2- s^2) } + s)^2 }{ (1-r^2)^2\sqrt{ 1-(r^2-s^2) }}. \label{phi_Berwald}
\ee
One can also construct a projectively flat metric with constant flag curvature $K=-1$ (\cite{Sh1}). $\phi=\phi(r,s)$ is given by
\be
\phi= {1\over 2} \Big \{ { \sqrt{1 -
(r^2-s^2) }+ s \over 1- r^2} - \e { \sqrt{1 -\e^2
(r^2-s^2) }+\e s \over 1- \e^2r^2}
\Big \},\label{phi_Shen}
\ee
where  $-1\leq \e < 1$ is a constant.  They are all spherically symmetric metrics with constant flag curvature.

Recently, Mo-Zhou-Zhu  finds three equations that characterize  spherically symmetric metrics of constant curvature and find some new locally projectively flat  metrics of constant flag curvature (\cite{Mo}).  In this paper, we shall show  that these three equations can be reduced to two equations (Theorems \ref{thm1.2} and \ref{thm1.4} below).

It is also a natural problem to study  spherically  symmetric metrics with constant Ricci curvature.  We find one equation that characterizes spherically symmetric metrics of constant Ricci curvature (Theorem \ref{thm1.1} below) and two equations that characterize  those of constant Ricci curvature tensor (Theorem \ref{thm1.3} below).

To state our results, we introduce the following notations.
 For a positive smooth function $\phi =\phi(r,s)$ on $[0, \rho)\times (-\rho, \rho)$, let
\begin{eqnarray*}
R_1: & = & P^2-\frac{1}{r}( sP_r+rP_s) + 2Q [ 1+sP+(r^2-s^2)P_s] \\
R_2: & = &  \frac{1}{r}(2Q_r-sQ_{rs}-rQ_{ss})+2Q(2Q-sQ_s) + (r^2-s^2)(2QQ_{ss}-[Q_s]^2) ,\\
R_3: & = & \frac{1}{r} ( P_r-sP_{rs}-rP_{ss}) + 2Q [ 1+ sP+(r^2-s^2) P_s]_s,
\end{eqnarray*}
where
\begin{eqnarray*}
P  &  =  &  \frac{1}{2r\phi} ( s\phi_r+r\phi_s) -\frac{Q}{\phi} \{ s\phi + (r^2-s^2) \phi_s \}\\
Q & = &  \frac{1}{2r} \frac{s\phi_{rs}+ r\phi_{ss}-\phi_{r}}{\phi-s\phi_{s} +(r^2-s^2)\phi_{ss}}.
\end{eqnarray*}
Note that for the above three functions, $\phi=\phi(r,s)$, in (\ref{phi_Funk}), (\ref{phi_Berwald}) and (\ref{phi_Shen}),  $Q=0$, i.e.,  $ s\phi_{rs}+ r\phi_{ss}-\phi_{r}=0.$  In this case, $R_i's$ can be simplified further.

In this paper, we shall prove the following

\begin{thm} \label{thm1.1} Let $F = |y| \phi (r,s )$ be a spherically symmetric Finsler metric where $r:=|x| $ and $s:=\frac{\langle x, y \rangle}{|y|}$.
Then $Ric=(n-1)K F^2$ ($K=constant$)    if and only if   $\phi$ satisfies the  PDE below:
\begin{equation}
 (n-1)K\phi^2=  (n-1)R_1 + (r^2-s^2)   R_2.  \label{Eqn-MainThm}
\end{equation}
\end{thm}

In (\ref{Eqn-MainThm}), if $R_2=0$, then (\ref{Eqn-MainThm}) is reduced to that $K\phi^2=R_1$. In this case, $F$ is actually of constant flag curvature $K$.

\begin{thm}\label{thm1.2}
Let $F = |y| \phi(r,s)$ be a \ssfm ~ on an open ball  $B^n(\rho)\subset R^n$ ($n\geq 3$).   Then $F$ is of \cfc ~$K$ if and only if
\be
R_1 = K \phi^2, \ \ \ \  \ R_2=0, \label{R1R2}
\ee
\end{thm}

Theorem \ref{thm1.2} improves a result in \cite{Mo} (see Theorem \ref{MO.10.13-thm1.1x} below).

There is a notion of Ricci curvature tensor $Ric_{ij}$ introduced   in \cite{LiSh1}.
\be
Ric_{ij} := \frac{1}{2} \Big \{   R^{\ m}_{i \ mj} + R^{\ m}_{j \ mi}\Big \}, \label{Ric_ij}
\ee
where $R^{\ i}_{j \ kl}$ denotes the Riemann curvature tensor of the Berwald connection.
Note that
\be
 Ric = Ric_{ij} y^iy^j.\label{RicRic_ij}
\ee
 By (\ref{RicRic_ij}), we see that $Ric_{ij} =(n-1) K g_{ij}$ implies that $Ric = (n-1) K F^2$.  We have the following

\begin{thm}\label{thm1.3}
  Let $F = |y| \phi (r,s )$  be a spherically symmetric Finsler metric where $r:=|x| $ and $s:=\frac{\langle x, y \rangle}{|y|}$.
Then
 $Ric_{ij}=(n-1)K g_{ij}$ ($K=constant$)   if and only if   $\phi$ satisfies
\be
(n-1) K \phi^2 =(n-1)R_1+(r^2-s^2)R_2, \ \ \ \  (n+1) R_3 +(r^2-s^2) [R_2]_s=0.  \label{thm1.2eq}
\ee
\end{thm}

Assume that (\ref{thm1.2eq}) holds. If  $R_2=0$, then $R_3=0$ and $K\phi^2=R_1$, thus $F$ is of constant flag curvature. In fact that the condition $R_2=0$ and $R_3=0$ is sufficient for $F$ to be of constant flag curvature.  We have the following

\begin{thm}\label{thm1.4}
Let $F = |y| \phi(r,s)$ be a \ssfm ~ on an open ball  $B^n(\rho)\subset R^n$ ($n\geq 3$).   Then $F$ is of constant flag curvature if and only if
\[    R_2 =0, \ \ \ \ \ R_3=0.\]
\end{thm}


\section{Preliminaries}
 Let $F = |y| \phi (r,s )$  be a spherically symmetric Finsler metric  on $R^n$, where $r:=|x| $ and $s:=\frac{\langle x, y \rangle}{|y|}$.
According to  Mo-Zhou-Zhu (\cite{Mo}),   the Riemann curvature tensor  $R^i_{\ j} $ is given by

\begin{align} \label{Rij}
R^i_{\ j} &=   R_1 ( |y|^2\delta^i_j -y^iy^j) +  |y| R_2 ( |y| x^j-sy^j)x^i   + R_4 (|y|x^j -s y^j ) y^i,
\end{align}
where  $R_1, R_2, R_3$ are given in the introduction above and $R_4$ is given by

\begin{align} \label{Req1}
R_4  &:=   \frac{1}{2} \{ 3 R_3 - [R_1]_s\},
\end{align}

Note that $R_3$ here is not the $R_3$ in \cite{Mo} and $R_1, R_2, R_4$ are the same terms as in \cite{Mo}.
Using the identity $g_{ik} R^i_{\ j} = g_{ij}R^i_{\ k} $, we  immediately obtain

\begin{align} \label{Req2}
R_4 + s R_2 &=    -\frac{\phi_s}{\phi}  \Big \{ (r^2-s^2) R_2 + R_1 \Big \},
\end{align}

Recall  that  $F$ is  of scalar flag curvature if and only if $R^i_{\ j} = R \delta^i_j - \tau_j y^i$ with $\tau_j y^j = R$. Thus it is easy to see from (\ref{Rij}) that $F = |y| \phi(r, s)$ is of scalar flag curvature if and only if $R_2 =0$.

\begin{lem}\label{Lemma2.1} {\rm (\cite{HM})} Let $F = |y| \phi(r,s)$ be a \ssfm ~ on a ball  $B^n(\rho)\subset R^n$ ($n\geq 3$). Then $F$ is of scalar flag curvature if and only if $R_2=0$.
\end{lem}

By the above formula (\ref{Rij}), Mo-Zhou-Zhu \cite{Mo}  prove the following
%
%
\begin{thm} (\cite{Mo}) \label{MO.10.13-thm1.1x}
 Let $F = |y| \phi(r,s)$ be a \ssfm ~ on a ball  $B^n(\rho)\subset R^n$ ($n\geq 3$). Then $F$ is of \cfc ~$K$ if and only if

\begin{align}
R_{1}   & =      K \phi^{2}, \label{CC-eq1x} \\
R_{2}   & =    0,   \label{CC-eq3x}\\
R_{3}   & =   0,  \label{CC-eq2x}
\end{align}
where $R_1, R_2$ and $ R_3$ are given as above.
\end{thm}

In fact, two equations (\ref{CC-eq1x}) and (\ref{CC-eq3x}) in Theorem \ref{MO.10.13-thm1.1x} will be sufficient (Theorem \ref{thm1.2} above).

\bigskip
There is an important non-Riemannian quantity, $\chi=\chi_i dx^i$,   defined by the S-curvature \cite{Sh}.
\[  \chi_i :=\frac{1}{2}\Big \{  S_{\cdot i |m}y^m-S_{|i}  \Big \} ,\]
where $S$ denotes the S-curvature of $F$  with respect to the Busemann-Hausdorff volume.
It can be also expressed in terms of the Riemann curvature  $R^i_{\ k}= R^{\ i}_{j\ kl}y^jy^l$  by

\begin{align}\label{chi_iR}
\chi_i  & =     -\frac{1}{6} \Big \{  2R^m_{\ i\cdot m} + R^m_{\ m\cdot i} \Big \},
\end{align}

where $``\cdot ''$ denotes the vertical covariant  derivative.  The importance of this $\chi$-curvature lies in the following

\begin{lem} \label{Lemma2.3}
{\rm (\cite{Sh})} For a Finsler metric of scalar flag curvature on an $n$-dimensional manifold, $\chi_i=0$ if and  only if the flag curvature is isotropic (constant if $n\geq 3$).
\end{lem}

Let  $F=|y|\phi(r, s)$ be a spherically symmetric metric on $B^n(\rho)\subset R^n$.
By differentiating (\ref{Rij}) and using (\ref{chi_iR}), we can easily obtain a formula for $\chi_i$:
\begin{equation}
\chi_i
= -\frac{1}{2}\Big  \{ (n+1) R_3 + (r^2-s^2) [R_2]_s \Big \} ( |y| x^i - s y^i ).
\label{chiR1R2R4}
\end{equation}
We have the following

\begin{lem}
\label{Lemma2.4}
For a spherically symmetric metric on $R^n$,  $\chi_i=0$ if and only if
\be
 (n+1) R_3 + (r^2-s^2) [R_2]_s =0.
\ee
\end{lem}

There is another important non-Riemannian quantity, the  H-curvature  $H= H_{ij} dx^i\otimes dx^j$,  defined by $ H_{ij}:=  E_{ij|m}y^m$,  where $E_{ij}:= \frac{1}{2} S_{\cdot i \cdot j}$ denotes  the mean Berwald curvature.  Here $S$ is the S-curvature.   $H$ can be also expressed in terms of $\chi_i$ by
\be
  H_{ij} =\frac{1}{2}\Big  \{ \chi_{i\cdot j}+\chi_{j\cdot i} \Big \}.\label{Hijchi}
\ee
(See \cite{Sh}).
The Ricci curvature tensor  $Ric_{ij}$ in (\ref{Ric_ij}) is related to the Ricci curvature  $Ric = R^m_{\ m}$  by the following identity:
\be
Ric_{ij} = \frac{1}{2} [Ric]_{y^iy^j} + H_{ij}.
\ee

For spherically symmetric metrics on $R^n$,
by  differentiating  $\chi_i$ and using (\ref{chiR1R2R4})  we obtain
\be
H_{ij} = M_s  |y|^{-2} ( |y| x^i -s y^i) (|y|x^j-s y^j) - s M |y|^{-2} (|y|^2 \delta_{ij} -y^i y^j ), \label{HijM}
\ee
where $M:=-\frac{1}{2}\{ (n+1) R_3 + (r^2-s^2) [R_2]_s\}$.
\\
\\
We have the following:

\begin{lem}\label{Lemma2.5} For spherically symmetric metrics on $R^n$,
$\chi_i=0$ if and only if $H_{ij}=0$.
\end{lem}
\begin{proof}
Assume that $H_{ij}=0$. Contracting (\ref{HijM}) with $x^i$ and $x^j$ yields
\[   H_{ij}x^ix^j = \{ (r^2-s^2)M_s - s M\} (r^2-s^2)=0.\]
Thus $(r^2-s^2)M_s = s M$. Plugging it into (\ref{HijM}) gives
\[ H_{ij} = M_s |y|^{-2} \{  ( |y| x^i -s y^i) (|y|x^j-s y^j) -(r^2-s^2) (|y|^2 \delta_{ij} -y^i y^j ) \} =0.\]
Clearly we have $M_s=0$, hence $M=0$. Then $\chi_i= M (|y| x^i-s y^i) =0$. This proves the lemma.
\end{proof}

\bigskip

\section{Proof of Main Theorems}

With the above preparation, the proofs of the main results become quite simple.

\bigskip

\begin{proof}(Theorem \ref{thm1.1})   We take the trace of the formula (\ref{Rij}). The trace is the Ricci curvature  given by

\begin{align}\label{Ric1}
Ric & =(n-1) |y|^2 R_1  + (r^2-s^2) |y|^2  R_2 .
\end{align}
Thus $Ric = (n-1)K F^2$ if and only if (\ref{Eqn-MainThm}) holds.
\end{proof}

\bigskip


\begin{proof}(Theorem \ref{thm1.2})   Assume that (\ref{R1R2}) holds.
Since $R_2=0$, we see that $F$ is of scalar flag curvature by Lemma \ref{Lemma2.1}.
On the other hand,
\[  Ric = (n-1)R_1 |y|^2 + (r^2-s^2) R_2  |y|^2  = (n-1)  K \phi^2 |y|^2 = (n-1) K F^2.\]
Namely, $F$ is of constant Ricci curvature $K$. Thus $F$  must be of constant flag curvature $K$. In this case, $R_3=0$ by Theorem \ref{MO.10.13-thm1.1x}.
This completes the proof.
\end{proof}

\bigskip

\begin{proof}(Theorem \ref{thm1.3})  For any Finsler metric,  $Ric_{ij}=(n-1) K g_{ij}$ if and only if $Ric=(n-1) K F^2$ and $H_{ij}=0$. By Lemma \ref{Lemma2.5}, for any spherically symmetric metric,  $H_{ij}=0$ if and only if $\chi_i=0$.  Thus for a spherically symmetric metric $F= |y|\phi(r,s)$ on $R^n$,
$Ric_{ij}= (n-1) K g_{ij}$ if and only if $Ric=(n-1) K F^2$ and $\chi_i=0$. By (\ref{chiR1R2R4}) and (\ref{Ric1}), we prove the theorem.
\end{proof}

\bigskip

\begin{proof}(Theorem \ref{thm1.4})  Assume that $F$ is of constant flag curvature. Then it follows from Theorems \ref{thm1.2} and \ref{thm1.3} that $R_2=0$ and $R_3 =0$. Conversely, assume that $R_2=0$ and $R_3 =0$. First by (\ref{chiR1R2R4}), we see that $\chi_i=0$. By Lemma \ref{Lemma2.1}, we see that $F$ is of scalar flag curvature. Then the theorem follows from Lemma \ref{Lemma2.3}.   We can also  prove this  using (\ref{Req1}) and (\ref{Req2}).  Under the assumption that $R_2=0$ and $R_3 =0$, and from (\ref{Req1}) and (\ref{Req2}), we get that
\[ -\frac{1}{2} [R_1]_s = R_4 =  -\frac{\phi_s}{\phi} R_1.\]
Thus
\[  \Big [ \frac{R_1}{\phi^2} \Big ]_s =0.\]
This gives
\[ R_1 = K \phi^2,\]
where $K= K(r)$ is independent of $s$. Then $F$ is of isotropic flag curvature by Theorem \ref{thm1.2}. $K$ must be a constant by the Schur Lemma.
\end{proof}

\section{Special Solutions}

We now look at the special case when $Q=0$, i.e.,
\be
 \phi_r - s\phi_{rs}- r\phi_{ss} =0.  \label{PDE1}
\ee
In this case, $F=|y| \phi(r, s)$ must be projectively flat and
\begin{eqnarray}
R_1  & = &  \psi^2 - \frac{1}{r} ( s \psi_r + r \psi_s)\\
R_2 & = &0,\\
R_3 & = & \frac{1}{r}\Big  \{ \psi_r-   s\psi_{rs} - r \psi_{ss} \Big \},\\
R_4 & = & \frac{1}{r} ( 2 \psi_r -r \psi\psi_s - s\psi_{rs}-r \psi_{ss} ),
\end{eqnarray}
where
\[ \psi: =\frac{1}{2r\phi}(s\phi_{r} +r \phi_{s}).\]
By Theorems \ref{thm1.2} and \ref{thm1.4}, $F$ is of constant flag curvature $K$ if $\phi$ satisfies one of the following equations:
\be
K \phi^2 = \psi^2 - \frac{1}{r} (r\psi_{s} + s \psi_{r}),     \label{pde-M2}
\ee
\be
\psi_r - s\psi_{rs} - r \psi_{ss}  =0.  \label{PDE2}
\ee
Note that (\ref{PDE1}) and (\ref{PDE2}) are similar.

 (\ref{PDE1}) and (\ref{pde-M2}) are solvable
(see Shen-Yu \cite{ShenYu}), hence we obtain special solutions of (\ref{R1R2}).
 The key idea is to use the following special substitution
\[   u : = r^2 -s^2, \ \ \ \ \ v : =s.\]
Then
\[  \phi_r = 2r \phi_u, \ \ \ \ \ \ \phi_s = -2s \phi_u +\phi_v.\]
This gives
\[ \psi = \frac{ \phi_v}{2\phi} = \Big (\ln \sqrt{\phi} \Big )_v.\]
Similarly, we have
\[ \frac{ s\psi_r + r\psi_s} {r} = \psi_v.\]
Thus (\ref{pde-M2}) can be written as
\be
K\phi^2 = \psi^2 -\psi_v = \Big ( \frac{\phi_v}{2\phi}  \Big )^2 - \Big ( \frac{\phi_v}{2\phi}\Big )_v.
\ee
This is just an ODE in $v$. After solving  this ODE, then plugging it into (\ref{PDE1}), one
obtains  all solutions to (\ref{PDE1}) and (\ref{pde-M2}).
The corresponding spherical symmetric metrics  must be of  constant curvature.

\begin{prop}  {\rm (\cite{ShenYu})}  \label{lem3} The non-constant solutions of equations  (\ref{PDE1}) and  (\ref{pde-M2}) are given by

\begin{align}
\phi(r,s)& = \frac{1}{2\sqrt{-K}}  \frac{1}{ \sqrt{C-r^2+s^2 }+ s}   \label{soln1}  ~~~~~~~~ \texttt{or} \\
\phi(r,s)& = \frac{q}{q^2( D q+v)^2 +K}  \label{soln2}     \\
\nonumber\\
 ~~ \texttt{where} ~~q& \neq   0 ~~  \texttt{is determined by the following equation}~~\nonumber\\
\nonumber\\
0& = D^2q^4 +(u-C) q^2-K.    \label{soln2-c1}  
\end{align}

where $u=r^2-s^2$, $v=s$, and both $C$ and $D$  are  constant numbers.
\end{prop}

Three interesting solutions are given as follows

\begin{description}
\item[Solution 1]   $D\neq0 $  and  $K=0$, $\phi$ is given by
\begin{align}
\phi(r,s)& = \frac{D}{\sqrt{C-r^2 +s^2}(\sqrt{C-r^2 +s^2} - s)^2} \label{soln4}
\end{align}
\noindent In this case, the corresponding \ssfm s are Berwald metrics.\\

\item[Solution 2]  $D\neq0 $  and  $K=-1$, $\phi$ is given by
\begin{align}
\phi(r,s)& = \frac{1}{2}\Big\{\frac{1}{\sqrt{C+2D-r^2 +s^2}-s}  -\frac{1}{\sqrt{C-2D-r^2 +s^2}-s} \Big\}   \label{soln5}
\end{align}
\noindent In this case, the corresponding \ssfm s are first given by Z. Shen in \cite{Sh1}.\\

\item[Solution 3]   $D\neq0 $, $K=1$,   and  $q$ is real, $\phi$ is given by
\begin{align}
\phi(r,s)& =  Re\Big (\frac{1}{\sqrt{C+2iD-r^2 +s^2}-is}  \Big) \label{soln6}
\end{align}
\noindent In this case, the corresponding \ssfm s are Bryant's metrics.
\end{description}

\subsection*{Acknowledgements}
Dr. \d{S}engelen Sev\.{i}m  and Dr. \"Ulgen  were partly supported by The Scientific and Technological Research Council of Turkey (TUBITAK) (grant no. 113F311).

\end{document}